\documentclass[a4paper,oneside]{amsart}

\usepackage[utf8]{inputenc}

\usepackage{amsmath,amsfonts,amssymb,amsthm}
\usepackage[charter]{mathdesign}

\usepackage{fullpage}
\usepackage{url}
\usepackage{latexsym}

\usepackage{mathrsfs} 
\usepackage[colorlinks, urlcolor=blue, citecolor=black, linkcolor=black ,breaklinks]{hyperref} 

\usepackage{enumerate}

\newtheorem{definition}{Definition}[section]
\newtheorem{theorem}[definition]{Theorem}
\newtheorem{lemma}[definition]{Lemma}
\newtheorem{corollary}[definition]{Corollary}
\newtheorem{proposition}[definition]{Proposition}

\newcommand{\Remark}[1][ ]{{\bf Remark. #1 }}

\def\N{\mathbb{N}}
\def\F{\mathbb{F}}
\def\P{\mathsf{P}}
\def\Div{\mathsf{Div}}
\def\A{\mathcal{A}}
\newcommand{\Fq}{\mathbb{F}_q}

\def\FF{\mathbf{F}}
\def\GG{\mathbf{G}}
\def\HH{\mathbf{H}}
\def\O{\mathcal{O}}
\def\i{\mathrm{i}}

\def\XX{\mathcal{C}}
\def\B{\mathrm{B}}

\newcommand{\Cl}{\mathsf{Cl}}
\newcommand{\Jac}{\ensuremath{\mathsf{Jac}}}
\newcommand{\Ld}[2][D]{\ensuremath{\mathscr{L}(#2\mathcal{#1})}} 
\newcommand{\D}[1][D]{\ensuremath{\mathcal{#1}}} 

\newcommand{\mus}{\mu^\mathrm{sym}}
\newcommand{\Ms}{\mathrm{M}^\mathrm{sym}}

\begin{document}

\title[Tower with maximal Hasse-Witt invariant and tensor rank over $\F_{2^n}$ and $\F_{3^n}$]{Tower of algebraic function fields with maximal Hasse-Witt invariant \\and tensor rank of multiplication in any extension of $\F_2$ and $\F_3$}
\author{St\'ephane Ballet}
\address{Aix Marseille Universit\'e, {CNRS}, Centrale Marseille, {I2M}, {UMR 7373}, 13453 Marseille, 
France\qquad\linebreak[4]\indent Case 907, 163 Avenue de Luminy, 13288 Marseille Cedex 9, France}
\email{stephane.ballet@univ-amu.fr}
\author{Julia Pieltant}
\address{Inria -- Saclay \^Ile-de-France,
LIX, \'Ecole Polytechnique,
91128 Palaiseau Cedex, France}
\email{pieltant@lix.polytechnique.fr}
\urladdr{http://www.lix.polytechnique.fr/~pieltant}
\date{\today}

\begin{abstract}

Up until now, it was recognized that a large number of 2-torsion points was a technical barrier to improve the bounds for the symmetric tensor rank of multiplication in every extension of any finite field. In this paper, we show that there are two exceptional cases, namely the extensions of $\F_2$ and $\F_3$. In particular, using the definition field descent on the field with 2 or 3 elements of a Garcia-Stichtenoth tower of algebraic function fields which is asymptotically optimal in the sense of Drinfel'd-Vl\u{a}du\c{t} and has maximal Hasse-Witt invariant, we obtain a significant improvement of the uniform bounds for the symmetric tensor rank of multiplication in any extension of $\F_2$ and $\F_3$.

\end{abstract}

\maketitle


\section{Introduction}

\subsection{General context} The determination problem of the tensor rank of multiplication in finite fields has been widely studied over the past  20 years. This problem is worthwhile both because of its theoretical interest and because it has several applications in the area of information theory such as cryptography and coding theory. In particular, Shparlinski, Tsfasman and Vl\u{a}du\c{t} have developed a correspondence between bilinear multiplication algorithms and linear codes with good parameters \cite{shtsvl}. Their work is an achievement of the brilliant idea introduced by D.V. and G.V. Chudnovsky in \cite{chch}. 

The theory of bilinear complexity of multiplication is a part of algebraic complexity theory. For a more extensive presentation of the background and the framework of this topic, we refer the reader to the classic book \cite{buclsh} by B\"urgisser, Clausen and Shokrollahi.

\subsection{Tensor rank and multiplication algorithm}

Let us recall the notions of multiplication algorithm and associated bilinear complexity as in \cite{rand3}:

\begin{definition}
Let $K$ be a field and $E_0, \ldots , E_s$ be finite dimensional $K$-vector spaces. A non zero element \linebreak[4]${t\in E_0 \otimes \cdots \otimes E_s}$ is said to be an elementary tensor, or a tensor of rank 1, if it can be written in the form \linebreak[4]${t=e_0 \otimes \cdots \otimes e_s}$ for some ${e_i \in E_i}$. More generally, the rank of an arbitrary ${t\in E_0 \otimes \cdots \otimes E_s}$  is defined as the minimal length of a decomposition of $t$ as a sum of elementary tensors.
\end{definition}

\begin{definition}
If 
$$
\alpha\; : \; E_1\times \cdots \times E_s \longrightarrow E_0
$$
is an $s$-linear map, the $s$-linear complexity of $\alpha$ is defined as the tensor rank of the element
$$
\tilde{\alpha}\in E_0\otimes E_1^{\vee} \otimes \cdots \otimes E_s^{\vee}
$$
naturally deduced from $\alpha$; where $E_i^{\vee}$ denotes the dual of $E_i$ as vector space over~$K$ for any integer $i$. In particular, the $2$-linear complexity is called the bilinear complexity.
\end{definition}

\begin{definition}
Let $\A $ be a finite-dimensional $K$-algebra. We denote by 
$$\mu(\A/K)$$
the bilinear complexity of the multiplication map
$$
\mathsf{m}_{\A} \; : \; \A \times \A \longrightarrow \A
$$
considered as a $K$-bilinear map. 

\noindent
In particular, if $\A=\F_{q^n}$ and $K=\Fq$,  we set:
$$
\mu_q(n):=\mu(\F_{q^n}/\Fq).
$$
\end{definition}

More concretely, $\mu(\A/K)$ is the smallest integer $\ell$ such that there exist linear forms \linebreak[4]${\phi_1,\ldots , \phi_\ell, \psi_1, \ldots  , \psi_\ell  :  \A \longrightarrow K}$, and elements ${w_1, \ldots , w_\ell \in \A}$, such that for all ${x,y\in \A}$ one has
\begin{equation} \label{xy}
xy= \mathsf{m}_\A(x,y) = \phi_1(x)\psi_1(y)w_1+\cdots +\phi_\ell(x)\psi_\ell(y)w_\ell,
\end{equation}
since such an expression is the same thing as a decomposition
\begin{equation}
t_{\mathsf{m}_\A}=\sum_{i=1}^\ell w_i\otimes \phi_i\otimes \psi_i \in  \A \otimes \A^\vee \otimes \A^\vee
\end{equation}
for the multiplication tensor of $\A$.

\begin{definition}
We call multiplication algorithm of length $\ell$  for $\A/K$ a collection of $\phi_i, \psi_i, w_i$ that satisfy \eqref{xy} or equivalently a decomposition 
$$
t_{\mathsf{m}_\A}=\sum_{i=1}^\ell w_i\otimes \phi_i\otimes \psi_i \in  \A \otimes \A^\vee \otimes \A^\vee
$$ 
for the multiplication tensor of $\A$. Such an algorithm is said symmetric if ${\phi_i=\psi_i}$ for all $i$ (this can happen only if $\A$ is commutative).
\end{definition}

Hence, when $\A$ is commutative, it is interesting to study the minimal length of a symmetric multiplication algorithm since it turns out that it plays an important role in several other areas such as Riemann-Roch system of equations, arithmetic secret sharing, multiplication-friendly codes, etc, as mentioned in \cite{cacrxing3}.

\begin{definition}
If $\A$ is a finite-dimensional $K$-algebra, the symmetric bilinear complexity
$$
\mus(\A/K)
$$ 
is the minimal length of a symmetric multiplication algorithm. 

\noindent
In particular, if $\A=\F_{q^n}$ and $K=\F_q$,  we set:
$$
\mus_q(n):=\mus(\F_{q^n}/\F_q).
$$
\end{definition}

\subsection{Basic notions related to function fields and notation} 

Let $\FF/\Fq$ be an algebraic function field of one variable of genus~$g$, with constant field $\Fq$, associated to a curve $\XX$ defined over $\F_q$. In the sequel, we may  simultaneously use the dual language of (smooth, absolutely irreducible, projective) curves by associating to $\FF/\Fq$ a unique ($\Fq$-isomorphism class of) curve $\XX/\Fq$ of genus~$g$ and conversely to such a curve its function field. 

For any integer ${k\geq 1}$, we denote by ${\P_k(\FF/\Fq)}$ the set of places of degree $k$, by ${\mathrm{B}_k(\FF/\Fq)}$ the cardinality of this set and by ${\P(\FF/\Fq)=\cup_k \,\P_k(\FF/\Fq)}$ the set of all places in $\FF/\Fq$.

For any place $P$, we define $\mathrm{F}_P$ to be the residue class field of $P$ and $\O_P$ its valuation ring. Every element ${t \in P}$ such that ${P = t \O_P}$ is called a local parameter for $P$ and we denote by ${v_P}$ a discrete valuation associated to the place $P$ in $\FF/\Fq$. Recall that this valuation does not depend on the choice of the local parameter. 

The divisor group of ${\FF/\Fq}$  is denoted by ${\Div(\FF/\Fq)}$.
The degree of a divisor ${\D=\sum_P a_P P}$ is defined by \linebreak[4]${\deg \D=\sum_P a_P \deg P}$ where $\deg P$ is the dimension of $\mathrm{F}_P$ over $\Fq$; the support of $\D$ is the set ${\mathrm{supp} \, \D}$ of the places $P$ such that ${a_P  \neq 0}$;  the order of the divisor $\D$ in $P$ is the integer $a_P$ denoted by ${\mathrm{ord}_P \, \D}$.

We denote by $\Div_n(\FF/\Fq)$ the set of divisors of degree $n$ and we say that the divisor $\D$ is effective if for each ${P\in \mathrm{supp}\, \D}$, we have $a_P \geq 0$. We denote by ${\Div_n^+(\FF/\Fq)}$ the set of effective divisors of degree $n$ and we set ${\mathrm{A}_n := \# \Div_n^+(\FF/\Fq)}$.
  Let $f \in \FF/\Fq$ be non-zero, we denote by $(f)$ the divisor
associated to the function $f$, namely
$$
(f)=\sum_{P\in\P(\FF/\Fq)} v_P(f)P.
$$
Such a divisor $(f)$ is called a principal divisor, and the set of  principal divisors is denoted by ${\mathsf{Princ}}(\mathbf{F}/\mathbb{F}_q)$; it is a subgroup of ${\Div_0(\FF/\Fq)}$.
 Two divisors $\D_1$ and $\D_2$ are said to be equivalent, denoted by ${\D_1 \thicksim \D_2}$, if ${\D_1=\D_2+(f)}$ for an element ${f \in \FF\backslash \{0\}}$.
 The factor group 
$$
\Cl(\FF/\Fq)= \Div(\FF/\Fq)/\mathsf{Princ}(\FF/\Fq)
$$
is called the divisor class group.
We will denote by $[\D]$ the class of the divisor $\D$ in $\Cl(\FF/\Fq)$.

For any divisor $\D$, the Riemann-Roch space associated to $\D$ is the set
$$
\Ld{}=\{f \in \FF/\Fq \; | \; \D + (f) \geq 0\}\cup \{0\}.
$$
It is a vector space over $\Fq$ whose dimension is denoted $\dim \D$.
If $\D_1 \thicksim \D_2$, the following holds:
$$
\deg\D_1 = \deg \D_2, \qquad \qquad \dim\D_1=\dim\D_2,
$$
so that we can define the degree $\deg\,[\D]$ and the dimension $\dim\,[\D]$ of a class.\\
Since the degree of a principal divisor is zero, we can define the subgroup
${\Cl_0(\FF}/\Fq)$ of classes of degree zero divisors in ${\Cl(\FF/\Fq)}$.
It is a finite group and we denote by $h(\FF/\Fq)$ its order, called the class number of $\FF/\Fq$.\\
Moreover if 
$$
L_\FF(t)=\sum_{i=0}^{2g} a_i t^i=\prod_{i=1}^{g} (1-\pi_i t)(1-\overline{\pi}_i t)
$$
 is the numerator of the Zeta function of $\FF/\Fq$, where ${|\pi_i|=\sqrt{q}}$, then we have ${h(\FF/\Fq)=L_\FF(1)}$.

By F. K. Schmidt's Theorem (cf. \cite[Corollary V.1.11]{stic}), there always exists a rational divisor of degree one, so the group ${\Cl_0(\FF/\Fq)}$ is isomorphic to the group of 
$\Fq$-rational points on the Jacobian of $\XX$, denoted by $\Jac(\XX)$. 
In particular, $h(\mathbf{F}/\mathbb{F}_q)=\# \Jac(\XX)(\Fq)$. \\

The Riemann-Roch Theorem states that the dimension of the vector space $\Ld{}$ is related to the degree of the divisor $\D$ and to the genus of $\FF/\Fq$:
\begin{equation} \label{RR}
\dim \D = \deg \D-g+1 +\dim (\kappa-\D),
\end{equation}
where $\kappa$ denotes a canonical divisor of $\FF/\Fq$ (or equivalently  a divisor of degree ${2g-2}$ and dimension $g$).
In this relation, the complementary term ${\i(\D) :=\dim (\kappa-\D)}$ is called the index of speciality of $\D$. Note that in any case, we have ${\i(\D)\geq0}$.
In particular, a divisor $\D$ is called a non-special divisor when the index of speciality $\i(\D)$
is zero and is called a special divisor if ${\i(\D)>0}$. 
Many deep results have been obtained on the study of non-special divisors in the 
dual language of curves when the field of definition is algebraically closed. 
See for instance \cite{arcoal} 
for a beautiful survey over $\mathbb{C}$. On the contrary, few results are known when 
the rationality of the divisor is taken into account as in our context where 
we require the divisor $\D$ to be defined over $\mathbb{F}_q$.
We refer to~\cite{balb} for known results on the existence of non-special divisors of degree $g$ and ${g-1}$.

\subsection{Known results}

\subsubsection{General results}\label{sect:intro}

The bilinear complexity $\mu_q(n)$ of the multiplication in the $n$-degree extension of a finite field $\F_q$ is known for certain values of $n$.  In particular, S. Winograd \cite{wino3} and H. de Groote \cite{groo} have shown that this complexity is ${\geq 2n-1}$, with equality holding if and only if ${n \leq \frac{1}{2}q+1}$. 
Using the principle of the D.V. and G.V. Chudnovsky algorithm \cite{chch} applied to elliptic curves, M.A. Shokrollahi has shown in \cite{shok} that the bilinear complexity of multiplication is equal to $2n$ for ${\frac{1}{2}q +1< n < \frac{1}{2}(q+1+{\epsilon (q) })}$ where $\epsilon$ is the function defined by:
$$
\epsilon (q) = \left \{
	\begin{array}{l}
 		 \mbox{greatest integer} \leq 2{\sqrt q} \mbox{ prime to $q$, if $q$ is not a perfect square} \\
  		2{\sqrt q}\mbox{, if $q$ is a perfect square.}
	\end{array} \right .
$$

Moreover, U. Baum and M.A. Shokrollahi have succeeded in \cite{bash} to construct effective optimal algorithms of Chudnovsky-Chudnovsky  type  in the elliptic case. 

Recently in \cite{ball1}, \cite{ball3}, \cite{baro1}, \cite{balbro}, \cite{balb}, \cite{bach} and \cite{ball5} the study made by M.A. Shokrollahi has been  generalized to algebraic function fields of arbitrary genus. \\

Let us recall that the original algorithm of D.V. and G.V. Chudnovsky introduced in \cite{chch} leads to the following theorem:

\begin{theorem}
Let $q$ be a prime power. The tensor rank $\mu_q(n)$ of multiplication in any finite extension~$\F_{q^n}$ of $\Fq$ is linear with respect to the extension degree; more precisely, there exists a constant $C_q$ such that for any $n$, it holds that:
$$
\mus_q(n) \leq C_q n.
$$
\end{theorem}

Moreover, one can give explicit values for $C_q$:

\begin{proposition} \label{theobornes}
Let $q$ be a power of the prime $p$. The best known values for the constant $C_q$ defined in the previous theorem are:
$$
C_q = \left \{
	\begin{array}{lll}
		\hbox{if } q=2 & \hbox{then \ } 19.6 &   \hbox{\cite{ceoz,bapi}} \cr
		\hbox{else if } q=3 &  \mbox{then \ } 27 & \hbox{\cite{ball1}} \cr
		\hbox{else if } q=p \geq 5 &  \mbox{then \ }   3(1+ \frac{4}{q-3}) & \hbox{\cite{bach}} \\
		\hbox{else if } q=p^2 \geq 25 & \hbox{then \ }   2(1+\frac{2}{\sqrt{q}-3}) & \hbox{\cite{bach}} \\
		\hbox{else if } q=p^{2k} \geq 16 & \hbox{then \ } 2(1+\frac{p}{\sqrt{q}-3}) & \hbox{\cite{ball3}} \\
		\hbox{else if } q \geq 16 & \hbox{then \ }  3(1+\frac{2p}{q-3}) & \hbox{\cite{baro1,balbro,balb}}\\
		\hbox{else if } q > 3 & \hbox{then \ }  6(1+\frac{p}{q-3}) & \hbox{\cite{ball3}}.
	\end{array}\right .
$$
\end{proposition}
\Remark The estimate $C_2=19.6$ is obtained by combining the general uniform bound \linebreak[4]${\mus_2(n) \leq \frac{477}{26} n + \frac{45}{2}}$ from \cite{bapi} for $n$ greater than 19, and the values of $\mus_2(n)$ given in \cite[Table 1]{ceoz} for~$n\leq 18$.\\

In order to obtain these good estimates for the constant $C_q$, S. Ballet has given in \cite{ball1} some easy to verify conditions allowing the use of the D.V. and G.V. Chudnovsky  algorithm. Then S. Ballet and R. Rolland have generalized in \cite{baro1} the algorithm using places of degree one and two. Let us present the best finalized version of this algorithm in this direction, which is a generalization of the algorithm of Chudnovsky-Chudnovsky type introduced by N. Arnaud in \cite{arna1} and developed later by M. Cenk and F. \"Ozbudak in \cite{ceoz}. This generalization uses several coefficients in the local expansion at each place $P_i$ instead of just the first one. Due to the way to obtain the local expansion of a product from the local expansion of each term, the bound for the symmetric bilinear complexity involves the complexity notion $\widehat{M_q}(u)$ introduced by M. Cenk and F. \"Ozbudak in \cite{ceoz} and defined as follows:
\begin{definition}
We denote by $\widehat{M_q}(u)$ the minimum number of multiplications needed in $\F_q$ in order to obtain coefficients of the product of two arbitrary $u$-term polynomials modulo $x^u$ in $\F_q[x]$.
\end{definition}
Remark  that with the notations introduced in Section 1, one has ${\widehat{M_q}(u)=\mu\Big(\left(\F_q[x]/ (x^u)\right) /\F_q\Big)}$.\\
For instance, we know that for all prime powers $q$, we have $\widehat{M_q}(2) \leq 3$ by \cite{ceoz2}.\\

Note that in \cite{rand3}, Randriambololona gives an even more general version of the Chudnovsky-Chudnovsky algorithm, which encompass the case of non-necessarily symmetric algorithms. This generalization is not relevant here, since we focus on the symmetric bilinear complexity; thus we introduce the generalized symmetric algorithm Chudnovsky-Chudnovsky type described in~\cite{ceoz}.

\begin{theorem} \label{theo_evalder}
Let
\begin{itemize}
	\item $q$ be a prime power,
	\item $\FF/\F_q$ be an algebraic function field,
	\item $Q$ be a degree $n$ place of $\FF/\F_q$,
	\item $\D$ be a divisor of $\FF/\F_q$,
	\item ${\mathscr P}=\{P_1,\ldots,P_N\}$ be a set of $N$ places of arbitrary degree,
	\item ${t_1, \ldots, t_N}$ be local parameters for ${P_1, \ldots,P_N}$ respectively,
	\item ${u_1,\ldots,u_N}$ be positive integers.
\end{itemize}
We suppose that $Q$ and all the places in $\mathscr P$ are not in the support of ${\mathcal D}$ and that:
\begin{enumerate}[a)]
	\item the map
	$$
	 \mathsf{Ev}_Q:  \left |
	\begin{array}{ccl}
	\Ld{} & \rightarrow & \F_{q^n}\simeq \mathrm{F}_Q\\
	f & \longmapsto & f(Q)	
	\end{array} \right.
	$$ 
	is onto,
	\item the map
	$$
	 \mathsf{Ev}_{\mathscr P} :  \left |
	\begin{array}{ccl}
	\Ld{2} & \longrightarrow & \left(\F_{q^{\deg P_1}}\right)^{u_1} \times \left(\F_{q^{\deg P_2}}\right)^{u_2} \times \cdots \times \left(\F_{q^{\deg P_N}}\right)^{u_N} \\
	f & \longmapsto & \big(\varphi_1(f), \varphi_2(f), \ldots, \varphi_N(f)\big)
	\end{array} \right.
	$$
	is injective, where each application $\varphi_i$ is defined by
	$$
	\varphi_i : \left |
	\begin{array}{ccl}
	\Ld{2} & \longrightarrow & \left(\F_{q^{\deg P_i}}\right)^{u_i} \\
          f & \longmapsto & \left(f(P_i), f'(P_i), \ldots, f^{(u_i-1)}(P_i)\right)
	\end{array} \right.
 	$$
	with $f = f(P_i) + f'(P_i)t_i + f''(P_i)t_i^2+ \ldots + f^{(k)}(P_i)t_i^k + \ldots $, 
the local expansion at $P_i$ of $f$ in ${\Ld{2}}$, with respect to the local parameter~$t_i$. 
Note that we set ${f^{(0)} :=f}$.
\end{enumerate}
Then 
$$
\mus_q(n) \leq \displaystyle \sum_{i=1}^N \mus_q(\deg P_i) \widehat{M_{q^{\deg P_i}}}(u_i).
$$
\end{theorem}

In particular, we will consider in this paper a specialization of this algorithm which is described in Section~\ref{sect:newbound} and requires the additional hypothesis that there exists a non-special divisor of degree ${g-1}$; this will motivate the study of ordinary towers.

\subsubsection{Asymptotic bounds for the extensions of $\F_2$ and $\F_3$}

From the asymptotic point of view, let us recall that I. Shparlinski, M. Tsfasman and S. Vl\u{a}du\c{t} have given in \cite{shtsvl} many interesting remarks on the algorithm of D.V. and G.V. Chudnovsky. In particular, they considered the following asymptotic bounds for the bilinear complexity
$$
\mathrm{M}_q = \limsup_{k \rightarrow \infty} \frac{\mu_q(k)}{k}
$$
and we will also consider its symmetric equivalent
$$
\Ms_q = \limsup_{k \rightarrow \infty} \frac{\mus_q(k)}{k}.
$$

Recently, with the help of the torsion-limit technique and Riemann-Roch systems, Cascudo, Cramer and Xing improved in \cite{cacrxing3} the upper bounds for $\Ms_q$ in the case where $q$ is small (${q \in \{2,3,4,5\}}$). In particular, they obtained:
$$
\Ms_2 \leq 7.23  \qquad \mbox{and} \qquad \Ms_3 \leq 5.45.
$$

\subsection{Motivations -- New results established in this paper}
Contrary to what is mentioned in \cite{cacrxing3} by Cascudo, Cramer and Xing, and in \cite{babe1} by Bassa and Beelen, we will show that an ordinary tower may lead to better uniform results for the tensor rank  of multiplication in any extension of $\F_2$ and $\F_3$ than a non-ordinary one because of the link between maximal $p$-rank and existence of a non-special divisor of degree ${g-1}$. Indeed, in \cite{babe1}, it reads:\textsl{
``A detailed study of the $p$-rank in towers is relevant for their applications, see \cite{cacrxing0}. For example, although both of the towers introduced in \cite{gast3} and \cite{gast} have the same limit and hence are equally influential for applications in coding theory, a detailed study of their $p$-rank reveals that in fact the latter \textsf{(which turns out not to be ordinary, according to \cite{cacrxing3})} is more appropriate for other kinds of applications, e.g. secure multiparty computation and fast bilinear multiplication.''}

We know that the existence of a non-special divisor of degree ${g-1}$ in the function field ${\FF/\Fq}$ is of crucial importance in the performance of Chudnovsky-Chudnovsky type algorithms \cite{balb,rand3}. In the case where the definition field $\Fq$  is such that ${q\geq4}$, then according to \cite{balb} there always exists a non-special divisor of degree ${g-1}$. Nevertheless, the problem persists in the case where the definition field is small, namely $\F_2$ or $\F_3$. In \cite{bapi}, to avoid this obstacle, we substituted \mbox{non-special} divisors of degree ${g-1}$ for zero-dimensional divisors whose degree is as close as possible to ${g-1}$ in the descent over $\F_2$ of the original Garcia-Stichtenoth tower presented in \cite{gast} and defined over $\F_{16}$; non-special divisors of degree ${g-1}$ being the borderline case of zero-dimensional divisors. However, according to a result of Bassa and Beelen in  \cite{babe1}, the second optimal Garcia-Stichtenoth tower introduced in \cite{gast3} is ordinary. But it was shown in \cite{bariro} that there always exists a non-special divisor of degree $g-1$ in any ordinary function field .
This leads us to an improvement of the bounds for the uniform tensor rank of multiplication in any finite extension of $\F_2$ and $\F_3$, thanks to the existence of a non-special divisor of degree ${g-1}$ in any function field of some ordinary towers defined respectively over $\F_2$ and $\F_3$, as it will be proven in this paper.
In particular, we prove that
$$
\mus_2(n) \leq 16.16 n\qquad \mbox{and} \qquad \mus_3(n) \leq 7.732 n,
$$
which improves the results obtained in \cite{bapi}. Note that the difficulty to obtain non-asymptotic estimations of the 2-torsion points in all steps of the tower used in \cite{cacrxing3} is an obstruction to obtain uniform bounds as we get in this paper.

\section{Definitions and related properties of the $p$-rank}

\begin{definition}\label{defrank}
The $p$-rank $\gamma(\FF)$, also called invariant de Hasse-Witt, of a function field $\FF$ with constant field $\overline{\F_p}$, the algebraic closure of the finite field $\F_p$, is defined as the dimension over $\F_p$ of the group of divisor classes of degree zero of order $p$. If the function field is defined over a finite field $\F_q$, we define its $p$-rank as the $p$-rank of the function field $\FF\overline{\Fq}$, obtained by extending the constant field to the algebraic closure of $\Fq$. \end{definition}

It can be shown that :

\begin{proposition}
If $\FF/\Fq$ be a function field of genus $ g(\FF)$, then ${0 \leq \gamma(\FF) \leq g(\FF)}$.
\end{proposition}

\begin{definition}
A function field $\FF/\Fq$ is called ordinary if ${\gamma(\FF) = g(\FF)}$.\\
A tower of function fields  ${\mathcal{T} = \big(\FF_n/\F_q\big )_{n\in \N}}$ is said ordinary if for any~${n\geq0}$, $\FF_n$ is such that ${\gamma(\FF_n) = g(\FF_n)}$, i.e. if any step of the tower is an ordinary function field.
\end{definition}

Let us recall the following result from \cite{bariro}:

\begin{corollary} \label{ordinary}
If $\FF$ is a function field of genus ${g>0}$ defined over $\F_2$ or over $\F_3$, then there is always a  degree ${g-1}$ zero-dimensional divisor in $\FF$.
\end{corollary}

Moreover, directly from Definition \ref{defrank}, we can deduce the following lemma:

\begin{lemma}\label{conservationord}
Let ${r\geq 0}$ be an integer. If we set ${\FF/\F_{q^r} := \mathbf{H}/\Fq \otimes \F_{q^r}}$, then ${\mathbf{H}/\Fq}$ is ordinary if and only if ${\FF/\F_{q^r}}$ is ordinary.
\end{lemma}

\begin{proof}
Note that the genus does not change under constant field extension or descent.
It follows from Definition \ref{defrank} that  $p$-rank does not change under constant field extension 
or descent since the $p$-rank of a function field $\FF/\F_q$ defined over a finite field $\F_q$ is equal to 
the $p$-rank of $\FF\overline{\F_q}$, and  $\overline{\F_q}=\overline{\F_p}$.
\end{proof}

To conclude this section, we recall the following result  which is proven in \cite[Lemma 6, 2.]{babe1}:

\begin{lemma}\label{conservationorddesc}
If $\HH/\FF$ is a finite extension of function fields with same constant field $\Fq$, then 
$$
g(\HH) - \gamma(\HH) \geq g(\FF) - \gamma(\FF).
$$
In particular, if $\HH$ is ordinary then so is $\FF$.
\end{lemma}

\section{Good ordinary sequences of function fields defined over $\F_2$ or $\F_3$} \label{asympexactsequence}

In this section, we present sequences of algebraic function fields defined over $\F_2$ or $\F_3$, constructed from the well-known Garcia-Stichtenoth tower defined in \cite{gast3}, which will be used to obtain new bounds for the tensor rank of multiplication.

\subsection{Definition of  Garcia-Stichtenoth's towers} \label{sectiondefinition}

Let  us consider a finite field $\F_{q^2}$ with ${q=p^r}$, for $p$ a prime number and $r$ an integer. 
We consider the Garcia-Stichtenoth's elementary abelian tower $T_0$ over $\F_{q^2}$ constructed in \cite{gast3} 
and defined by the sequence ${(\FF_0, \FF_1,\FF_2,\ldots )}$ where 
$$
\FF_0 := \F_{q^2}(x_0)
$$
is the rational function field over $\F_{q^2}$, and for any $i \geq 0$, ${\FF_{i+1} := \FF_i(x_{i+1})}$ with $x_{i+1}$ satisfying the following equation:  
$$
x^q_{i+1}+x_{i+1} = \frac{x^q_i}{x_i^{q-1} + 1}.
$$
\noindent Let us denote by $g_i$ the genus of $\FF_i$ in $T_0/\F_{q^2}$ and recall the following formul\ae:
\begin{equation}\label{genregs}
g_i = \left \{ \begin{array}{ll}
		(q^{\frac{i+1}{2}}-1)^2  & \mbox{for odd } i, \\
		(q^{\frac{i}{2}}-1)(q^{\frac{i+2}{2}}-1) & \mbox{for even } i. 
		\end{array} \right .
\end{equation}

\noindent Thus, according to these formul\ae, it is straightforward that the genus of any step of the tower satisfies:
\begin{equation}\label{bornegenregs}
(q^{\frac{i}{2}}-1)(q^\frac{i+1}{2}-1) < g(\FF_i) < (q^{\frac{i+2}{2}}-1)(q^\frac{i+1}{2}-1).
\end{equation}
Moreover, a tighter upper bound will be useful and can be obtained by expanding expressions in (\ref{genregs}):
\begin{equation}\label{bornesupgenregs}
g(\FF_i) \leq q^{i+1}-2q^\frac{i+1}{2}+1.
\end{equation}

If the characteristic $p=2$ and $r=2$, i.e. ${q = 4}$, then one can densify the Garcia-Stichtenoth's tower with steps defined over the finite field $\F_{q^2}$ 
by considering the following completed tower:
$$
T_1/\F_{16} \ : \qquad \FF_{0,0}\subseteq \FF_{0,1}\subseteq \FF_{0,2} = \FF_{1,0}\subseteq \FF_{1,1} \subseteq \FF_{1,2} = \FF_{2,0} \subseteq \cdots
$$ 
such that ${\FF_i \subseteq \FF_{i,s} \subseteq \FF_{i+1}}$ for any integer ${s \in \{0, 1, 2\}}$, with ${\FF_{i,0}:=\FF_i}$ and ${\FF_{i,2}:=\FF_{i+1}}$. 
Indeed:

\begin{proposition}\label{T3}
There exists a tower $T_1$ defined over $\F_{16}$ whose recursive equation is defined over $\F_2$. More precisely,  the tower $T_1$ is the densified Garcia-Stichtenoth's tower over $\F_{16}$ and is defined by ${T_1 =\big(\FF_{i,s}\big)_{\begin{subarray}{l}\i\geq0\\ s\in{\{0,1\}}\end{subarray}}}$ where for any ${i\geq0}$:
$$
\FF_{i,0} := \FF_i \qquad \mbox{ and } \qquad \FF_{i,1}:= \FF_i(t_{i+1})
$$
with $t_{i+1}$ satisfying the equation:
\begin{equation}\label{eq:intermed}
t_{i+1}^2+t_{i+1} = \frac{x^4_i}{x_i^{3} + 1} \quad \mbox{ for } i=0, \ldots, n-1.
\end{equation}
\end{proposition}

\begin{proof}

Let $x_0$ be a transcendental element over $\F_{2}$ and let us set
$$\FF_0:=\F_{16}(x_0).$$
We define recursively for $i \geq 0$
\begin{enumerate}[(i)]
	\item $x_{i+1}$ such that  $x_{i+1}^4+x_{i+1}=\frac{x^4_i}{x_i^{3} + 1}\  \mbox{ for } i=0, \ldots, n-1$,
 	\item $t_{i+1}$ such that $t_{i+1}^2+t_{i+1} = \frac{x^4_i}{x_i^{3} + 1}\  \mbox{ for } i=0, \ldots, n-1$
(or alternatively $t_{i+1}=x^2_{i+1}+x_{i+1}$).
\end{enumerate}
Thus, we can define recursively the tower $T_1$ by setting:
$$
\FF_{i,1}=\FF_{i,0}(t_{i+1})=\FF_i(t_{i+1}) \qquad \mbox{ and } \qquad \FF_{i+1,0}=\FF_{i+1}=\FF_i(x_{i+1}).
$$  
\end{proof}

Let us remark that it is possible to densify the general Garcia-Stichtenoth's tower over $\F_{q^2}$ for any characteristic $p$ and for any integer $r$ since  each extension $\FF_{i+1}/\FF_i$ is Galois of degree $q=p^r$ with full constant field~$\F_{q^2}$.
However, in the general case the equation (\ref{eq:intermed}) for the intermediate steps is not defined over $\F_p$ but over $\F_q$. 
For example, for ${p=3}$ and ${r=2}$, we obtain an equation which is defined over $\F_9$.

\vspace{.5em}
 
\paragraph{\textbf{Notation.}} In the sequel, we will denote by ${\B_k(\FF/K)}$  the number of places of degree~$k$ of an algebraic function field $\FF/K$ defined over a finite field $K$; we will also denote by $g_{i,s}$ the genus of $\FF_{i,s}/\F_{16}$ in $T_1/\F_{16}$.

\subsection{Descent of the definition field of a Garcia-Stichtenoth's tower on the fields $\F_2$ and $\F_3$} \label{sectiondescent}
First we state that when $q=3$, one can descend the definition field of the tower $T_0/\F_{q^2}$ from $\F_{q^2}$ to $\F_q$ since the recursive equation defining the tower has coefficients lying in $\F_q$. Thus, we have the following result:
\begin{proposition}\label{proptourq3}
If $q=p=3$, there exists a tower $E/\F_q$ defined over $\F_{q}$ given by a sequence:
$$
\GG_{0} \subseteq \GG_{1} \subseteq \GG_{2} \subseteq  \GG_3 \subseteq \cdots 
$$
defined over the constant field $\F_q$ and related to the tower $T_0/\F_{q^2}$ by
$$
\FF_{i} = \F_{q}\GG_{i}\ \  \mbox{ for all }i,
$$
namely $\FF_{i}/\F_{q^2}$ is the constant field extension of $\GG_{i}/\F_q$. 
\end{proposition}

Now, we are interested in the descent of the definition field of the tower $T_1/\F_{q^2}$ 
from $\F_{q^2}$ to  $\F_{p}$ if it is possible. In fact, for the tower $T_1/\F_{q^2}$, one can not establish a general result 
but one can prove that it is possible in the case where the characteristic is $2$ and $r=2$, i.e. $q=4$. 
Note that in order to simplify the presentation, we are going to set the results by using the variable $p$. 

\begin{proposition}\label{proptourp2}
If $p=2$ and $q=p^2$, the descent of the definition field of the tower $T_1/\F_{q^2}$ from $\F_{q^2}$ to  $\F_{p}$ is possible. 
More precisely,  there exists a tower $T_2/\F_p$ given by a sequence:
$$
 \HH_{0,0} \subseteq \HH_{0,1} \subseteq \HH_{0,2} = \HH_{1,0} \subseteq \HH_{1,1}\subseteq \HH_{1,2} = \HH_{2,0} \subseteq \cdots
$$
defined over the constant field $\F_p$ and related to the tower $T_1/\F_{q^2}$ by
$$
\FF_{i,s}=\F_{q^2}\HH_{i,s}\  \  \mbox{ for all $i\geq0$ and $s\in \{0,1,2\}$},
$$
namely $\FF_{i,s}/\F_{q^2}$ is the constant field extension of $\HH_{i,s}/\F_p$.
\end{proposition}

\begin{proof}
It is a straightforward consequence of Proposition \ref{T3}.
\end{proof}

In order to draw consequences for the previously descended towers, let us recall the known results concerning the number of places of degree one of the tower $T_0/\F_{q^2}$, established in \cite{gast3} and \cite{alkushst}.

\begin{proposition}\label{genus}
If $q=p^r\geq2$, then for any $n>2$:
$$
\B_1(\FF_n/\F_{q^2}) = \left\{ 
	\begin{array}{ll}
	 q^n(q^2-q)+2q^2 & \mbox{if } p=2,\\
	 q^n(q^2-q)+2q & \mbox{if }  p>2.
\end{array}\right.
$$
\end{proposition}

Now, we deduce some straightforward properties concerning the towers $T_2/\F_2$ and $E/\F_3$. 

\begin{proposition}\label{subfieldp2}
Let $q=p^2=4$. For any integers ${i\geq0}$ and ${s \in \{0,1,2\}}$, the algebraic function field $\HH_{i,s}/\F_{p}$ in the tower $T_2/\F_p$  has ${\B_1(\HH_{i,s}/\F_p)}$ places of degree one, ${\B_2(\HH_{i,s}/\F_p)}$ places of degree two and ${\B_4(\HH_{i,s}/\F_p)}$ places of degree four and satisfies:
\begin{enumerate}[(i)]
	\item\label{inclus} ${\HH_i/\F_p \subseteq \HH_{i,s} /\F_p \subseteq \HH_{i+1}/\F_p}$ with ${\HH_{i,0}=\HH_i}$ and ${\HH_{i,2}=\HH_{i+1}}$,
	\item\label{bornegenreis} if ${g_{i,s}\stackrel{\mbox{def}}{:=}g(\HH_{i,s}/\F_p)}$ denotes the genus of $\HH_{i,s}/\F_p$, then:\\
	\begin{minipage}{0.35\linewidth}
	\begin{enumerate}[(a)]\label{eq:gis}
	\item[(ii.a)] \quad $g_{i,s} \leq \frac{g_{i+1}}{p^{2-s}}$
	\end{enumerate}
\end{minipage}
\hfill
\begin{minipage}{0.45\linewidth}
	\begin{enumerate}
	 \item[(ii.b)]\label{eq:gisindep} \quad $g_{i,s}\leq p^{s-2}(q^{i+2}-2q^{\frac{i}{2}+1})+p^{s-2}$\
	 \end{enumerate}
\end{minipage}
   
	\item\label{it:nbplaces} ${\B_{1}(\HH_{i,s}/\F_p)+2\B_{2}(\HH_{i,s}/\F_p) +4\B_{4}(\HH_{i,s}/\F_p)\geq q^{i}(q^2-q)p^{s}}$.
\end{enumerate}
Moreover, $\F_p$ is algebraically closed in each algebraic function field $\HH_{i,s}$ of the tower $T_2/\F_p$.
\end{proposition}

\Remark Bound (\ref{bornegenreis}.a) is tighter than Bound (\ref{bornegenreis}.b), but when we will need an estimate for $g_{i,s}$ which does not depend on the parity of the step $i$ of the tower, Bound (\ref{bornegenreis}.b) will be useful.

\begin{proof}
Property (\ref{inclus}) follows directly from Proposition \ref{proptourp2}. Each extension $\HH_{i+1}/\HH_{i,s}$ is a Galois extension of degree $[\HH_{i+1}:\HH_{i,s}]=2^{2-s}$. 
Moreover, the full constant field of $\HH_{i,s}$ is $\F_{p}$ since at least one place of $\HH_0$ 
is totally ramified in  $\HH_{i,s}$ by \cite[Prop. 3.7.8]{stic}.
Indeed, the place at infinity of $\FF_0$ is totally ramified in the tower $T_0/\F_{q^2}$. Hence, the same holds for the place at infinity of $\HH_0$ in $T_2/\F_{p}$. Since the algebraic function field $\FF_{i,s}$ is a constant field extension of $\HH_{i,s}$, for any integers ${i\geq0}$ and ${s\in\{0,1,2\}}$, $\FF_{i,s}$ and $\HH_{i,s}$ have the same genus, so by the Hurwitz Genus Formula \cite{stic}, we have: 
\begin{equation}\label{genrehurw}
g_{i,s} \leq \frac{g_{i+1}}{p^{2-s}}
\end{equation}
with $g(\HH_{i+1}/\F_p) = g(\FF_{i+1}/\F_{q^2})=g_{i+1}$ given by (\ref{genregs}). Finally, applying Bound (\ref{bornesupgenregs}) on $g_{i+1}$, we get 
(\ref{eq:gisindep}.b).
Moreover, for ${\alpha\in\F_{q^2}\backslash \{\omega \in \F_{q^2} \, \vert \, \omega^q + \omega = 0\}}$, let $P_{\alpha}$ denote the place of degree one in the rational function field $\FF_0$ which is the zero of $x_0-\alpha$, then $P_{\alpha}$ splits completely in $\FF_{i+1}/\FF_0$ by \cite[Lemma 3.9]{gast3}. Let us set $d:=[\FF_{i+1}:\FF_0]$ and $d':=[\FF_{i+1}:\FF_{i,s}]$. If $\ell$ denotes the number of places of $\FF_{i,s}$ lying over the place $P_{\alpha}$ of $\FF_0$, it is well known that $\ell \leq \frac{d}{d'}$ with equality holding if and only if $P_{\alpha}$ splits completely in  ${\FF_{i+1}/\FF_0}$. 
But, we also have ${d\leq \ell d'}$ which gives ${\ell\geq \frac{d}{d'}}$. It follows that ${\ell=\frac{d}{d'}= [\FF_{i,s}:\FF_0]}$ which proves that 
the place $P_{\alpha}$ splits completely also in $\FF_{i,s}/\FF_0$. Thus, there are exactly $q^{i}p^s$ places of degree one above 
$P_{\alpha}$ in $\FF_{i,s}$, so there are at least $q^{i}(q^2-q)p^s$ places of degree one in $\FF_{i,s}$, since ${\big\vert\,\F_{q^2}\backslash \{\omega \in \F_{q^2} \, \vert \, \omega^q + \omega = 0\} \, \big\vert= q^2-q}$.

To conclude, let us recall from \cite{gast3} that the number of places of degree one of $\FF_{i,s}/\F_{q^2}$ is such that \linebreak[4]${\B_1(\FF_{i,s}/\F_{q^2})\geq (q^2-q)q^{i}p^{s}}$. Thus, $\FF_{i,s}$ being a degree four constant field extension of $\HH_{i,s}$, it is clear that for any
integers $i\geq0$ and ${s\in\{0,1,2\}}$, it holds that
$$
\B_{1}(\HH_{i,s}/\F_p)+2\B_{2}(\HH_{i,s}/\F_p)+4\B_{4}(\HH_{i,s}/\F_p)\geq (q^2-q)q^{i}p^{s}.
$$
\end{proof}

Similar results than those of Proposition \ref{subfieldp2} can be obtained for the tower $E/\F_{3}$, namely:
\begin{proposition}\label{subfieldq3}
Let $q=p=3$. For any integer ${i\geq0}$, the algebraic function field $\GG_{i}/\F_{q}$ in the tower $E/\F_q$ has the same genus ${g_{i}}$ than the corresponding step $\FF_i/\F_{q^2}$ of the tower $T_0/\F_{q^2}$. Moreover, the number of places of degree one and two of each function field $\GG_i/\F_{q}$ is related to the number of rational places of ${\FF_{i}/\F_{q^2}}$ by: 
$$
\B_{1}(\FF_{i}/\F_{q^2}) = \B_1(\GG_i/\F_q)+2\B_{2}(\GG_{i}/\F_q) 
$$
thus, the following bound  holds:
\begin{equation}\label{eq:nbplaces12}
\B_1(\GG_i/\F_q)+2\B_{2}(\GG_{i}/\F_q) \geq q^{i}(q^2-q).
\end{equation}
\end{proposition}

To conclude this section, let us recall that in \cite{babe1}, the authors established the ordinarity of the classical tower over $\F_{q^2}$:
\begin{theorem}
For any prime power $q$, the tower ${T_0/\F_{q^2}}$  is ordinary.
\end{theorem}

Thus, we can deduce that the ordinarity of $T_0/\F_{q^2}$ provides the same property to the towers $T_2/\F_2$ and $E/\F_3$:

\begin{proposition}\label{prop_tourordi}
The towers ${T_2/\F_{2}}$ and $E/\F_3$ are ordinary.
\end{proposition}

\begin{proof}
Since constant field descent preserves ordinarity from Lemma \ref{conservationord}, the tower $E/\F_3$ is ordinary and so are the steps $\FF_{i,0}$ of the tower $T_2/\F_2$. Moreover Lemma \ref{conservationorddesc} implies that the intermediate steps $\FF_{i,1}$ are also ordinary since each one belongs to a finite extension ${\FF_{i+1,0}/\FF_{i,1}}$ with same constant field, where $\FF_{i+1,0}$ is ordinary.
\end{proof}

\begin{corollary}\label{cor:divnonspe}
For any function field $\FF$ in the towers ${T_2/\F_{2}}$ and $E/\F_3$, there exists a non-special divisor of degree~${g(F)-1}$.
\end{corollary}

\begin{proof}
It is a straightforward consequence of Corollary \ref{ordinary} and the last proposition.
\end{proof}


\section{New bounds for the tensor rank}\label{sect:newbound}

\subsection{Preliminary results}
To obtain our new estimates for $\mu_2(n)$ and $\mu_3(n)$ from the tower described in the previous section, we will need some technical results which are proven below.

\begin{theorem}\label{theo_bornesgenes}
Let $n$ and $d$ be two fixed integers. Let $\FF/\Fq$ be an algebraic function field of genus $g$ with at least $B_k$ places of degree $k$ for any ${k \vert d}$. If the three following conditions are satisfied:
\begin{enumerate}[(a)]
	\item\label{hyp_existdegn} ${\B_n(\FF/\Fq) > 0}$ (the inequality ${2g+1\leq q^\frac{n-1}{2}(\sqrt q-1)}$ is a sufficient condition),
	\item\label{hyp_divnonspe} there exists a non-special divisor of degree ${g-1}$,
	\item\label{hyp_nbplaces} $\displaystyle{\sum_{k \vert d}k(B_k+b_k)\geq 2n +2g-1}$, where the integers $b_k$ are chosen such that ${0 \leq b_k \leq B_k}$,
\end{enumerate}
then
$$
\mus_q(n) \leq \sum_{k\vert d} \mus_q(k)(B_k+b_k) +\sum_{k\vert d} \mus_q(k) b_k,
$$
so 
$$
\mus_q(n) \leq \eta \left(\sum_{k \vert d}k(B_k+b_k) + \sum_{k\vert d}k b_k\right) \qquad \mbox{ with } \eta := \max_{k\vert d} \frac{\mus_q(k)}{k}.\\
$$
\end{theorem}

\begin{proof}
The algorithm recalled in Theorem \ref{theo_evalder} is applied for a set ${\mathscr P =  \cup_{k \vert d} \mathscr{P}_k}$ with ${\mathscr{P}_k \subseteq \P_k(\FF/\F_q)}$ and ${\big \vert \mathscr{P}_k \big \vert = B_k}$. Among each ${P \in \mathscr{P}_k}$, $b_k$ are used with multiplicity ${u=2}$; all such places form a subset $\mathscr{R}$ of~$\mathscr{P}$. The others  ${B_k-b_k}$ places of $\mathscr{P}_k$ are used with multiplicity ${u=1}$. From the existence of a non-special divisor~$\D[G]$ of degree $g-1$ provided by Hypothesis (\ref{hyp_divnonspe}) and the existence of a place $Q$ of degree $n$, one constructs an effective divisor $\D$ such that $\deg \D = n+g-1$ and $\dim \D = n$. Precisely, one can choose any divisor which is equivalent to ${Q+\D[G]}$, but whose support is disjoint from the support  of ${Q+\D[G]}$. Then the following holds:
\begin{itemize}
	\item ${\ker \mathsf{Ev}_{Q} = {\mathcal L}(\D-Q) =\{0\}}$ since ${\D-Q \thicksim \D[G]}$ which is non-special of degree ${g-1}$ and so is zero-dimensional; thus  $\mathsf{Ev}_{Q}$ is bijective by dimension reasons,
	\item  ${\ker \mathsf{Ev}_{\mathscr P} =  {\mathcal L}\Big(2\D-\big(\sum_{P \in \mathscr{P} } P + \sum_{R \in {\mathscr R}}R \big)\Big) = \{0\}}$  from Hypothesis (\ref{hyp_nbplaces}) since $${\deg\left(2\D-(\sum_{P \in \mathscr{P}} P + \sum_{R \in {\mathscr R}}R )\right) = 2\deg \D - \sum_{k \vert d}k(B_k+b_k) < 0}$$ thus ${\mathsf{Ev}_{\mathscr P}}$ is injective.
\end{itemize}
As recalled in Section \ref{sect:intro},  $\widehat{M_q}(2) \leq 3$ so Theorem \ref{theo_evalder} then gives the following bound:
$$
\mus_q(n) \leq \sum_{P \in \mathscr{P} } \mus_q(\deg P) + 2\sum_{R \in {\mathscr R}} \mus_q(\deg R). 
$$
Rearranging summation to group places with the same degree, we get the result. 
\end{proof}

Here we state two special cases of Theorem \ref{theo_bornesgenes} which are adapted to the study of the tensor rank on $\F_2$ and $\F_3$ respectively.\\

This first one is adapted to the case where places of degree one, two and four are taking into account: 

\begin{proposition}\label{spec_chud4} Let $p=2$. If $\FF/\F_2$ is an algebraic function field of genus $g$ with at least $B_k$ places of degree $k$\linebreak[4]\mbox{for ${k=}$1, 2 and 4}, such that the three following conditions are satisfied:
\begin{enumerate}[(a)]
	\item ${\B_n(\FF/\F_2) > 0}$ (the inequality ${2g+1\leq p^\frac{n-1}{2}(\sqrt p-1)}$ is a sufficient condition),
	\item there exists a non-special divisor of degree ${g-1}$,
	\item\label{it:cond4} $\displaystyle{\sum_{k \vert 4}k(B_k+b_k)\geq 2n +2g-1}$, where the integers $b_k$ are chosen such that ${0 \leq b_k \leq B_k}$,
\end{enumerate}
then
$$
\mus_2(n) \leq \frac{9}{2}(n+g+1) + \frac{9}{4}\sum_{k \vert 4} k b_k.
$$
\end{proposition}

\begin{proof} It is a straightforward consequence of Theorem \ref{theo_bornesgenes}  with ${q=p=2}$ and ${d=4}$.
Recall that $\mus_2(2) = 3$ and $\mus_2(4) \leq 9$; so  ${\eta =  \max_{k\vert 4} \frac{\mus_2(k)}{k} \leq \max  \left \{ 1; \frac{3}{2} ; \frac{9}{4} \right\} = \frac{9}{4}}$ 
 and the result follows from a choice of the $B_k$'s and the $b_k$'s such that ${\sum_{k \vert 4}k(B_k+b_k)= 2n +2g-1+ \epsilon}$, with ${\epsilon \in \{0,1,2,3\}}$: we must consider the less favorable case where there only exists places of degree four and so we have to choose ${\epsilon=3}$.
\end{proof}

This second specialization corresponds to the case where only places of degree one and two are considered:
\begin{proposition}\label{spec_chud2} Let $q=3$. If $\FF/\F_3$ is an algebraic function field of genus $g$ with at least $B_k$ places of degree $k$\linebreak[4]\mbox{for ${k=}$1, 2} such that the three following conditions are satisfied:
\begin{enumerate}[(a)]
	\item ${\B_n(\FF/\F_3) > 0}$ (the inequality ${2g+1\leq q^\frac{n-1}{2}(\sqrt q-1)}$ is a sufficient condition),
	\item there exists a non-special divisor of degree ${g-1}$,
	\item\label{it:cond2} $\displaystyle{\sum_{k \vert 2}k(B_k+b_k)\geq 2n +2g-1}$, where the integers $b_k$ are chosen such that ${0 \leq b_k \leq B_k}$,
\end{enumerate}
then
$$
\mus_3(n) \leq 3(n+g) + \frac{3}{2}\sum_{k \vert 2} k b_k.
$$
\end{proposition}

\begin{proof} The same proof than the previous one with $q=3$ and $d=2$, and so $\eta = \frac{3}{2}$ in Theorem \ref{theo_bornesgenes} gives the result.
\end{proof}

\begin{lemma}\label{lem:Eetage2}
Let $q =4=p^2$ and ${n\geq 19}$. There exists a step $\HH_{i,s}/\F_2$ of the tower $T_2/\F_2$ such that the three conditions of Proposition \ref{spec_chud4} are satisfied with ${b_1=b_2=b_4=0}$. Moreover, if Condition (\ref{it:cond4}) is satisfied then the two others also are.
\end{lemma}

\begin{proof}  According to Corollary \ref{cor:divnonspe}, Condition (b) is satisfied for any step of the tower.\\
For ${i \leq \frac{n-13}{4}}$, it holds that ${p^{2i+6} \leq p^{\frac{n-1}{2}}}$. Then we get that ${p^{2i+6} \left(1-\frac{1}{p^{i+2}}+\frac{1}{p^{2i+3}}\right) \leq p^\frac{n-1}{2}p^2(\sqrt{p}-1)}$, since \linebreak[4]${1-\frac{1}{p^{i+2}}+\frac{1}{p^{2i+3}} \leq 1 \leq p^2(\sqrt{p}-1)}$. It follows that ${p^{2i+4}-p^{i+2}+p \leq p^\frac{n-1}{2}(\sqrt{p}-1)}$, which leads to\linebreak[4]${2g_{i,s}+1 \leq  p^\frac{n-1}{2}(\sqrt{p}-1)}$ according to Proposition \ref{subfieldp2} (\ref{eq:gisindep}.b) with ${s \in \{0,1\}}$ (one can always assume that ${s\neq 2}$ since ${\HH_{i,2} = \HH_{i+1,0}}$). Hence, Condition (a) is satisfied for any step $\HH_{i,s}$ such that ${i \leq \frac{n-13}{4}}$.
\\ On the other hand, for $i$ such that ${i > \log_q(n)-\frac{1}{2}}$, one has ${q^{i+1-\frac{1}{2}}\geq n+1}$, so ${q^{i+1}p^{s-1}(q-3) \geq n+1}$ since ${p^{s-1} \geq q^{-\frac{1}{2}}= p^{-1}}$, which gives ${q^{i+1}p^s(q-3) \geq 2n+2}$ and so ${q^{i+1}p^s(q-1-2) +q^{\frac{i-1}{2}p^s} \geq 2n+2}$. Thus, it holds that ${q^{i+1}p^s(q-1) \geq 2n+2 +2q^{i+1}p^s-q^\frac{i-1}{2}p^s = 2n +2p^{s-2}(q^{i+2}-q^{\frac{i}{2}+1})+2}$. Eventually, one gets that ${q^{i+1}p^s(q-1) \geq 2n +2p^{s-2}(q^{i+2}-q^{\frac{i}{2}+1})+2p^{s-2}-1}$ since ${2p^{s-2}-1 \leq 2}$ for ${s\in \{0,1\}}$, and Condition (c) is satisfied according to the inequalities (\ref{eq:gisindep}.b) and (\ref{it:nbplaces}) established in Proposition \ref{subfieldp2}. \\
Thus, for ${n \geq 21}$ one can find at least one integer $i$ in the interval ${\left] \log_q(n)-\frac{1}{2} ; \frac{n-13}{4} \right]}$, and so a corresponding step of the tower $\HH_{i,0}$ for which Proposition \ref{spec_chud4} holds. Note that in any case, Condition (a) is satisfied for lower steps than Condition (c), so it may happened that the first suitable step that satisfy both conditions is not $\HH_{i,0}$ itself but one of the previous step.
\\ Moreover one can check that for $n=19$, $\HH_1$ is the first suitable step of the tower to apply Proposition~\ref{spec_chud4} with ${b_1=b_2=b_4=0}$. Indeed, it holds that $g(\HH_1/\F_2) = 9$ so Condition (a) is satisfied and since $\B_1(\HH_1/\F_2) = 4$, $\B_2(\HH_1/\F_2) =2$ and $\B_4(\HH_1/\F_2) =12$, Condition (c) is also satisfied for $\HH_1$ but it is not the case for~$\HH_{0,1}$. Similarly for ${n =20}$,  $\HH_1$ does not satisfy Condition (c), but $\HH_{1,1}$ does satisfy both Conditions~(a) and~(c) since $g(\HH_{1,1}/\F_2) = 21$, $\B_1(\HH_{1,1}/\F_2) = 4$, $\B_2(\HH_{1,1}/\F_2) =2$ and $\B_4(\HH_{1,1}/\F_2) =25$.
\end{proof}

\begin{lemma}\label{lem:Eetage3}
Let $q=3$ and $n\geq 13$. There exists a step $\GG_{i}/\F_3$ of the tower $E/\F_3$ such that the three conditions of Proposition \ref{spec_chud2} are satisfied with ${b_1=b_2=0}$. Moreover, if Condition (\ref{it:cond2}) is satisfied then the two others also~are.
\end{lemma}

\begin{proof} According to Corollary \ref{cor:divnonspe}, Condition (b) is satisfied for any step of the tower.\\
For ${i\leq \frac{n-5}{2}}$, Condition (a) is satisfied since it holds that:
${q^i \leq \frac{q^\frac{n-4}{2}}{\sqrt{q}} \leq \frac{q^\frac{n-4}{2}}{\sqrt{q}+1} = q^\frac{n-4}{2}\frac{\sqrt{q}-1}{2}}$ and so\linebreak[4]${(q^\frac{i+2}{2}-1)(q^\frac{i+1}{2}-1) \leq q^\frac{n-1}{2}\frac{\sqrt{q}-1}{2}}$ which gives that ${2g_i+1 \leq q^\frac{n-1}{2}(\sqrt q-1)}$ according to (\ref{bornegenregs}).\\
On the other hand, when ${i \geq 2\log_q\left(\frac{n}{2}-1\right)}$, Condition (c) is satisfied. Indeed, for such $i$ one has: ${q^\frac{i}{2} \geq \frac{n}{2}-1}$, so ${4q^\frac{i}{2}\geq 2n-5}$, which gives that ${4q^\frac{i}{2}+2q\geq 2n+1}$. Adding ${2q^{i+1}}$, which equals ${(q^2-q)q^{i}}$, to both sides it follows that: ${(q^2-q)q^{i}+2q\geq 2n+2q^{i+1}-4q^\frac{i}{2}+1 = 2n+2(q^{i+1}-2q^\frac{i}{2}+1)-1}$. Thus from (\ref{eq:nbplaces12}) and (\ref{bornesupgenregs}) we get that inequality of Condition (c) holds with ${b_1=b_2=0}$.\\
To conclude, one can see that for $n \geq 13$, the interval ${\left[2\log_q\left(\frac{n}{2} -1\right) ; \frac{n-5}{2}\right]}$ contains at least an integer $i$ and so $\GG_i/\F_3$ is a suitable step of the tower; moreover the smallest such integer is the smallest  ${i \geq 2\log_q\left(\frac{n}{2}-1\right)}$, i.e. the smallest one for which Condition (c) is satisfied. 
\end{proof}

Till the end of this section, we will deal with the following notations: 
$$
\displaystyle{n_{2,i,s} \stackrel{\mbox{def}}{:=} \max \Big\{ m \, \big \vert \,  2m + 2g(\HH_{i,s})-1 \leq \sum_{k \vert 4} k \B_k(\HH_{i,s}/\F_2) \Big\}}
$$
and
$$
\displaystyle{n_{3,i} \stackrel{\mbox{def}}{:=} \max \Big\{ m \, \big \vert \,  2m + 2g(\GG_i)-1 \leq \sum_{k \vert 2} k \B_k(\GG_i/\F_3) \Big\}}.
$$

Let us explain the relevance of these definitions, focusing on the case of the role of $n_{3,i}$ in the tower $E/\F_3$ (the same holds for the tower $T_2/\F_2$ when one replaces $n_{3,i}$ by $n_{2,i,s}$).
The integer $n_{3,i}$ is the biggest one for which it holds that:
$$
\sum_{k \vert 2}k\B_k(\GG_i/\F_3)\geq 2n_{3,i} +2g_i-1
$$
 i.e. $\F_{q^{n_{3,i}}}$ is the biggest extension of $\F_3$ for which $\GG_i/\F_3$ could be a suitable step of the tower to apply Proposition \ref{spec_chud2} with ${b_1=b_2=0}$. If ${n >n_{3,i}}$, then 
 $$
\sum_{k \vert 2}k\B_k(\GG_i/\F_3) < 2n +2g_i-1
$$
but one has
$$
\sum_{k \vert 2}k\B_k(\GG_i/\F_3) + 2(n-n_{3,i})\geq 2n +2g_i-1
$$
which means that $\GG_i$ is still a suitable step of tower to  apply Theorem \ref{spec_chud2} if we can choose the $b_k$'s such that ${\sum_{k \vert 2} k b_k \geq 2(n-n_{3,i})}$.\\

Thus, we are interested in the determination of a lower bound for $n_{3,i}$ and $n_{2,i,s}$: it is the purpose of the two following lemmas:

\begin{lemma} If $p=2$ and $q=p^2=4$, then \ $n_{2,i,s} \geq  q^{i+1}p^s+q^{\frac{i}{2}+1}p^s-1.$
\end{lemma}

\begin{proof} According to Proposition \ref{subfieldp2} (\ref{it:nbplaces}) and (\ref{bornegenreis}.a), and Formula (\ref{bornesupgenregs}), we get:
\begin{eqnarray*}
	 \sum_{k \vert 4} k \B_k(\HH_{i,s}/\F_2) - 2g(\HH_{i,s})+1 & \geq & (q^2-q)q^ip^s-2p^{s-2}(q^{i+2}-2q^{\frac{i}{2}+1}+1) +1 \\
	 	& = & q^{i+2}p^s-q^{i+1}p^s -p^{s-1}(q^{i+2}-2q^{\frac{i}{2}+1}+1)+1\\
		& = & q^{i+2}p^{s-1}(p-1)-q^{i+1}p^s+q^{\frac{i}{2}+1}p^s-p^{s-1}+1\\
		& \geq & q^{i+1}p^{s-1}(q-p)+q^{\frac{i}{2}+1}p^s-1 \quad \mbox{ since } s \in \{0,1,2\} \mbox{ and } p-1 =1\\
		& = & q^{i+1}p^s+q^{\frac{i}{2}+1}p^s-1.
\end{eqnarray*}
\end{proof}

\begin{lemma} If $q=3$, then \ $ n_{3,i} \geq  4q^\frac{i+1}{2}-1$.
\end{lemma}

\begin{proof}
Proposition \ref{subfieldq3} and Formula~(\ref{bornesupgenregs}) give:
\begin{eqnarray*}
	 \sum_{k \vert 2} k \B_k(\GG_i/\F_3) - 2g(\GG_i)+1 & \geq & q^i(q^2-q)-2(q^{i+1}-2q^{\frac{i+1}{2}}+1)+1 \\
	 	& \geq & q^{i+1}(q-1)-2q^{i+1}+4q^\frac{i+1}{2} -1\\
		& = & q^{i+1}(q-3)+4q^\frac{i+1}{2}-1 = 4q^\frac{i+1}{2}-1.
\end{eqnarray*}
\end{proof}

Now, we establish a lower bound for the gap between the genus of two successive steps of each tower $T_2/\F_2$ and $E/\F_3$:

\begin{lemma}\label{lemme_deltag}
\begin{enumerate}[(i)]
	\item If ${p=2}$ and $q =p^2$, then
${
\Delta g_{i,s} \stackrel{\mbox{def}}{:=} g(\HH_{i,s+1})- g(\HH_{i,s}) \geq p^s(2q^i-3q^\frac{i}{2}).
}$\\
	
	\item If $p=q=3$ then 
${
\Delta g_{i} \stackrel{\mbox{def}}{:=} g(\GG_{i+1})- g(\GG_{i}) \geq (q-1)(q^{i+1}-q^{\lceil i/2\rceil}).
}$\\
\end {enumerate}
\end{lemma}

\begin{proof}
\begin{enumerate}[(i)]
	\item For any $s\in \{0,1\}$,  since ${[\HH_{i,s+1}: \HH_{i,s}] = p}$ the Hurwitz Genus Formula gives that \linebreak[4]${g_{i,s+1} - 1 \geq p(g_{i,s}-1)}$ and it follows that ${g_{i,s+1} - g_{i,s} \geq (p-1)(g_{i,s}-1)}$.\\
	If ${s=0}$, then ${g_{i,s}-1 = g_i -1}$ and according to  (\ref{bornegenregs}), it holds that  ${g_i -1 \geq (q^{\frac{i}{2}}-1)(q^\frac{i+1}{2}-1)}$. Thus, we get ${g_i \geq \sqrt{q}q^i -(1+\sqrt{q})q^\frac{i}{2} = 2q^i-3q^\frac{i}{2}}$, which gives that ${g_{i,s+1} - g_{i,s} \geq 2q^i-3q^\frac{i}{2} = p^s(2q^i-3q^\frac{i}{2})}$.\\
	If ${s=1}$, then ${g_{i,s+1} - g_{i,s} \geq (p-1)(g_{i,s}-1)}$ holds, with  ${g_{i,s}-1 = g_{i,1} -1\geq p(g_{i} -1)}$ from Hurwitz Genus Formula. Thus we get 
	${g_{i,s+1} - g_{i,s} \geq (p-1)p(g_i-1) \geq (p-1)p (2q^i-3q^\frac{i}{2})=p^s(2q^i-3q^\frac{i}{2})}$.
		\item From Formul\ae~(\ref{genregs}), we get 
	$$
	g_i = \left \{ \begin{array}{ll}
		(q-1)\left(q^{i+1}-q^\frac{i}{2} \right)  & \mbox{for even } i, \\
		(q-1)\left(q^{i+1}-q^\frac{i+1}{2} \right)  & \mbox{for odd } i, 
		\end{array} \right .
	$$
	which gives the result.
\end{enumerate}
\end{proof}

\subsection{Main results}

\begin{theorem}\label{newunifbounds} It holds that
$$\mus_2(n) \leq \underbrace{\frac{1035}{68}}_{\simeq 15.22}n + \frac{9}{2} \qquad and \qquad \mus_3(n) \leq \underbrace{\frac{1933}{250}}_{= 7.732}n.
$$
\end{theorem}

\begin{proof}
We first set $p=2$ and $q=p^2$. Note that for ${n \leq 18}$, the result already holds from Section~\ref{sect:intro} and~\cite[Table 1]{ceoz}. So, fix $n\geq 19$ and choose $i\geq0$ and ${s\in \{0,1\}}$ such that
$$
 \sum_{k \vert 4} k \B_k(\HH_{i,s+1}/\F_p) \geq 2n +2g_{i,s+1}-1 
$$
but 
$$
 \sum_{k \vert 4} k \B_k(\HH_{i,s}/\F_p) < 2n +2g_{i,s}-1.
$$
We can apply Proposition \ref{spec_chud4} in the two following ways:\\
\begin{enumerate}[(a)]
	\item on $\HH_{i,s+1}/\F_p$ with ${b_1=b_2=b_4=0}$, which gives:
	$$
	\mus_2(n) \leq \frac{9}{2}\left(n+ g_{i,s+1}+1\right)
	$$
	\item on $\HH_{i,s}/\F_p$ with the $b_k$'s chosen such that $\sum_{k \vert 4} k b_k := 2(n-n_{2,i,s})$ \ \textbf{if}  \ $2(n-n_{2,i,s}) \leq \sum_{k \vert 4} k \B_k(\HH_{i,s}/\F_p)$, which leads to:
	$$
	\mus_2(n) \leq \frac{9}{2}\left(n+ g_{i,s}+1\right) +  \frac{9}{4}\sum_{k \vert 4} k b_k.
	$$
\end{enumerate}
Rewriting those two bounds respectively as:
$$
\mus_2(n) \leq \frac{9}{2}(n-n_{2,i,s})+\frac{9}{2}(n_{2,i,s}+g_{i,s}+1)+ \frac{9}{2}\Delta g_{i,s}
$$
and 
$$
\mus_2(n) \leq 9(n-n_{2,i,s})+\frac{9}{2}(n_{2,i,s}+g_{i,s}+1)
$$
we see that the second one is better than the other as soon as ${n-n_{2,i,s} < \Delta g_{i,s}}$, under the assumption that ${2(n-n_{2,i,s}) \leq \sum_{k \vert 4} k \B_k(\HH_{i,s}/\F_p)}$. So if $D_{2,i,s}$ is such that ${D_{2,i,s} \leq \Delta g_{i,s}}$ and ${2D_{2,i,s} \leq \sum_{k \vert 4} k \B_k(\HH_{i,s}/\F_p)}$, then when ${n-n_{2,i,s} < D_{2,i,s}}$ , the second bound is better and can be reached since we can choose the $b_k$'s such that ${\sum_{k \vert 4} k b_k := 2(n-n_{2,i,s})}$. The particular case where ${n= n_{2,i,s}+D_{2,i,s}}$ will give us an upper bound for $\mus_2(n)$ as follows: define the function ${\Phi_2(x) := \min_{i,s}  \Phi_{2,i,s}(x) }$, with
$$
\Phi_{2,i,s}(x) = \left\{\begin{array}{ll}
				 9(x-n_{2,i,s})+\frac{9}{2}(n_{2,i,s}+g_{i,s}+1) &  \mbox{if } x- n_{2,i,s} < D_{2,i,s} \\
							 & \\
				\frac{9}{2}\left(x+ g_{i,s+1}+1\right) & \mbox{else},
			\end{array} \right.
$$
then $\mus_2(n)$ is bounded above by any linear function whose graph lies above all the points \linebreak[4]${\big\{\big(n_{2,i,s}+D_{2,i,s}, \Phi_p(n_{2,i,s}+D_{2,i,s})\big)\big\}_{i,s}}$.\\
We fix ${X := n_{2,i,s} + D_{2,i,s}}$ where
$$
D_{2,i,s} := \min \big \{ p^s(2q^i-3q^{\frac{i}{2}}) \, ; \, \frac{1}{2}q^i(q^2-q)p^s \big\} = p^s(2q^i-3q^\frac{i}{2})
$$
so that one has $D_{2,i,s} \leq \Delta g_{i,s}$ from Lemma \ref{lemme_deltag}, and ${D_{2,i,s} \leq \frac{1}{2} \sum_{k\vert 4} k\B_k(\HH_{i,s}/\F_p)}$ according to Theorem~\ref{subfieldp2}.
Thus, for any $i,s$, $\displaystyle{\Phi_2(X) \leq \frac{9}{2}\left(1+\frac{g_{i,s+1}}{X}\right)X+\frac{9}{2}}$.\\
One has
\begin{eqnarray*}
\frac{g_{i,s+1}}{X} &  \leq & \frac{p^{s}(q^{i+2}-3q^{\frac{i}{2}+1})+p^{s-1}}{ q^{i+1}p^s+q^{\frac{i}{2}+1}p^s-1 +p^s(2q^i-3q^\frac{i}{2})}\\ 
	& = &  \frac{q^{i+1}p^s(q-3q^\frac{i}{2}+q^{-i-1}p^{-1})}{q^{i+1}p^s(1+2q^{-1}+q^{-\frac{i}{2}}-3q^{-\frac{i}{2}-1}-q^{-i-1}p^{-s})}\\
	& = &  \frac{q-3p^i+q^{-i-1}p^{-1}}{1+2q^{-1}+p^{-i}-\underbrace{(3q^{-\frac{i}{2}-1}-q^{-i-1}p^{-s})}_{\leq 7/16}}\\
	& \leq &  \frac{q-3p^i+q^{-i-1}p^{-1}}{1+2q^{-1}+p^{-i}-\frac{7}{16}}\\
\end{eqnarray*}
which gives that \ $\displaystyle{\frac{g_{i,s+1}}{X} \leq \frac{81}{34}}$, so
$$
\mus_2(n) \leq \frac{9}{2}\left(1+\frac{81}{34}\right)n + \frac{9}{2} = \frac{1035}{68}n + \frac{9}{2}.\\
$$

Now we consider the case $q=p=3$. Since the result already holds  for $n<13$ from \cite[Table 1]{ceoz}, fix $n\geq 13$, 
 and choose $i\geq0$ such that
$$
 \sum_{k \vert 2} k \B_k(\GG_{i+1}/\F_q) \geq 2n +2g_{i+1}-1 
$$
but 
$$
 \sum_{k \vert 2} k \B_k(\GG_{i}/\F_q) < 2n +2g_{i}-1.
$$
We can apply Proposition \ref{spec_chud2} in the two following ways:
\begin{enumerate}[(a)]
	\item on $\GG_{i+1}/\F_q$ with ${b_1=b_2=0}$, which gives:
	$$
	\mus_3(n) \leq 3\left(n+ g_{i+1}\right)
	$$
	\item on $\GG_{i}/\F_q$ with the $b_k$'s chosen such that $\sum_{k \vert 2} k b_k := 2(n-n_{3,i})$ \ \textbf{if}  \ $2(n-n_{3,i}) \leq \sum_{k \vert 2} k \B_k(\GG_{i}/\F_q)$, which leads to:
	$$
	\mus_3(n) \leq 3(n+ g_{i})+ \frac{3}{2}\sum_{k \vert 2} k b_k.
	$$
\end{enumerate}
Rewriting those two bounds respectively as:
$$
\mus_3(n) \leq 3(n-n_{3,i})+3(n_{3,i}+g_{i})+ 3\Delta g_{i}
$$
and 
$$
\mus_3(n) \leq 6(n-n_{3,i})+3(n_{3,i}+g_{i})
$$
we see that the second one is better than the other when ${ n-n_{3,i}<\Delta g_{i}}$, under the assumption that \linebreak[4]${2(n-n_{3,i}) \leq \sum_{k \vert 2} k \B_k(\GG_{i}/\F_q)}$. So if $D_{3,i}$ is such that ${D_{3,i} \leq \Delta g_{i}}$ and ${2D_{3,i} \leq \sum_{k \vert 2} k \B_k(\GG_{i}/\F_q)}$, then when ${2(n-n_{3,i}) < D_{3,i}}$, the second bound is better and can be reached since we can choose the $b_k$'s such that ${\sum_{k \vert 2} k b_k := 2(n-n_{3,i})}$. The particular case where ${n= n_{3,i}+D_{3,i}}$ will give us an upper bound for $\mus_3(n)$ as follows: define the function ${\Phi_3(x) := \min_{i}  \Phi_{3,i}(x) }$, with
$$
\Phi_{3,i}(x) = \left\{\begin{array}{ll}
				 6(x-n_{3,i})+3(n_{3,i}+g_{i}) & \mbox{if } x- n_{3,i} < D_{3,i}\\
				 & \\
				 3\left(x+ g_{i+1}\right) & \mbox{else},
			      \end{array} \right.
$$
then $\mus_3(n)$ is bounded above by any linear function whose graph lies above all the points \linebreak[4]${\big\{\big(n_{3,i}+D_{3,i}, \Phi_3(n_{3,i}+D_{3,i})\big)\big\}_{i}}$.\\
We fix ${X := n_{3,i} +D_{3,i}}$ where
$$
D_{3,i} := \min \big \{ (q-1)(q^{i+1}-q^{\lceil i/2\rceil}) \, ; \, \frac{1}{2}q^i(q^2-q) \big\}.
$$
Thus, for any $i\geq2$, $D_{3,i}=(q-1)(q^{i+1}-q^{\lceil i/2\rceil})$; and it holds that $\displaystyle{\Phi_3(X) \leq 3\left(1+\frac{g_{i+1}}{X}\right)X}$.\\
One has:
\begin{eqnarray*}
\frac{g_{i+1}}{X} &  \leq & \frac{(q^{\frac{i+3}{2}}-1)(q^{\frac{i+2}{2}}-1)}{4q^\frac{i+1}{2}-1+(q-1)(q^{i+1}-q^{\lceil i/2\rceil})}\\ 
	& = &  \frac{q^{i+\frac{5}{2}} -q^\frac{i+2}{2}(1+\sqrt{q})+1}{q^{i+2} +4q^\frac{i+1}{2}-q^{i+1}-(q-1)q^{\lceil i/2\rceil}-1}\\
	& \leq & \frac{q^{i+2}\left(\sqrt{q}-q^{-\frac{i+2}{2}}(1+\sqrt{q})+q^{-i-2}\right)}{q^{i+2}\left(1+4q^{-\frac{i+3}{2}} -q^{-1}-(q-1)q^{-\frac{i+3}{2}}-q^{-i-2}\right)} 
\end{eqnarray*}
\noindent which gives that:
$$\displaystyle{\frac{g_{i+1}}{X} \leq \frac{\sqrt{q}-q^{-\frac{i+2}{2}}(1+\sqrt{q})+q^{-i-2}}{1 -q^{-1}-(q-1)q^{-\frac{i+3}{2}}-q^{-i-2}}}$$
so  since $i\geq2$:
$$\displaystyle{\frac{g_{i+1}}{X} \leq \frac{\sqrt{q}-q^{-2}(1+\sqrt{q})+q^{-4}}{1 -q^{-1}-(q-1)q^{-\frac{5}{2}}-q^{-4}}}.$$
Finally, with $q=3$, one gets:
$$
\mus_3(n) \leq \underbrace{3\left(1+\frac{\sqrt{3}-\frac{1}{9}(1+\sqrt{3})+3^{-4}}{\frac{2}{3}-2\cdot 3^{-\frac{5}{2}}-3^{-4}}\right)}_{\simeq 7.7314}n \leq  \underbrace{\frac{1933}{250}}_{=7.732}n.\\
$$
\end{proof}

\paragraph{\textbf{Remark.}} In the case of $\F_2$, the descent of the tower $T_0$ defined over $\F_{q^2}$ with $q=2$ from $\F_{q^2}$ to $\Fq = \F_2$ is not sufficient to obtain a competitive bound for the tensor rank. Indeed, in this case, we get:
$$
\mus_2(n) \leq  22.5n + \frac{9}{2}.
$$

\begin{corollary}
The following new estimates hold:
$$
C_2 = 16.16 \qquad \mbox{and} \qquad C_3 = 7.732.
$$
\end{corollary}

\begin{proof}
The estimate for $C_3$ is straightforward since $\frac{1933}{250} = 7.732$; for $C_2$, it follows from Proposition \ref{newunifbounds} for $n$ greater than 19 and \cite[Table 1]{ceoz} for $n\leq18$. 
\end{proof}

\end{document}